\documentclass[a4paper,10pt]{amsart}
\usepackage[utf8x]{inputenc}
\usepackage{amssymb,latexsym,amsmath,amsxtra,amsthm,amsfonts,amscd,slashed}
\usepackage{parskip}
\usepackage{enumerate}
\usepackage{color}
\usepackage{epsfig}
\usepackage{url}

\theoremstyle{plain}
\newtheorem{theorem}{Theorem}[section]
\newtheorem{corollary}[theorem]{Corollary}
\newtheorem{lemma}[theorem]{Lemma}
\newtheorem{proposition}[theorem]{Proposition}

\theoremstyle{remark}
\newtheorem{remark}{Remark}
\theoremstyle{definition}
\newtheorem{definition}[theorem]{Definition}

\makeatletter
\def\thm@space@setup{
  \thm@preskip=\parskip \thm@postskip=0pt
}
\makeatother

\newcommand{\lap}[1]{\Delta#1}

\newcommand{\ind}{\mathbf{1}}
\newcommand{\at}[2]{\left.#1\right|_{#2}}
\newcommand{\paren}[1]{\left(#1\right)}
\newcommand{\norm}[1]{\Vert#1\Vert}

\newcommand{\abs}[1]{\left|#1\right|}
\newcommand{\brac}[1]{\left[#1\right]}

\newcommand{\PD}[2]{\frac{\partial#1}{\partial#2}}
\newcommand{\vecs}[1]{\boldsymbol{#1}}
\newcommand{\pl}{p_{\lambda}}
\newcommand{\ext}[1]{E_0\brac{#1}}
\newcommand{\DO}[1]{\frac{d}{d#1}}
\newcommand{\PDDM}[3]{\frac{\partial^{2}{#1}}{\partial{#2}\partial{#3}}}
\newcommand{\pd}[2]{\partial_{#2}#1}

\def\R{\mathbb{R}}

\title[Large Coupling Limits in Sobolev Spaces]{Rate of Convergence for Large Coupling Limits in Sobolev Spaces}

\author{Ikemefuna Agbanusi}

\begin{document}

\maketitle

\begin{abstract}
We estimate the rate of convergence, in the so-called large coupling limit, for Schr\"odinger type operators on bounded domains with ``interaction potentials" supported in a compact inclusion. We show that if the boundary of the inclusion is sufficiently smooth, one essentially recovers the ``free Hamiltonian" in the exterior domain with Dirichlet boundary conditions. In addition, we obtain a convergence rate, in $L^2$, that is $\mathcal{O}(\lambda^{-\frac{1}{4}})$ where $\lambda$ is the coupling parameter. Our methods include energy estimates, trace estimates, interpolation and duality.
\end{abstract}

\section{Introduction and Overview}
This paper estimates the rate of convergence, as the coupling parameter gets arbitrarily large, for the Laplacian perturbed by a multiple of the characteristic function of a subdomain. This class of problems has been studied by many authors using different approaches. Our primary goal here is to show how one can obtain similar results using purely ``PDE techniques".

More concretely, let $\Omega\subset\R^m$, $m\geq3$, be a bounded open connected subset with smooth boundary $\Gamma$. Denote
by $A$ the self-adjoint realization of the Laplacian, $\lap$, in $L^2(\Omega)$ with Neumann boundary conditions on $\Gamma$. As is well known,
the operator $A$ generates a positive semigroup formally written as $e^{-tA}$. Intuitively, this semigroup corresponds to a diffusion process with reflection
at the boundary $\Gamma$.

Now, let $\Omega_0\Subset\Omega$ be a compact inclusion with boundary $\Gamma_0$ and put $ \Omega_1:=\Omega\backslash \overline{\Omega}_0$ as the \emph{exterior} (see Figure~\ref{fig:reaction_region}). We consider Schr\"odinger type operators of the form:
 $A_\lambda := A - \lambda\ind_{\overline{\Omega}_0}$; where $\lambda$ is a positive parameter and $\ind_{E}(x)$ is the characteristic function of 
the measurable set $E$. Such operators also determine (parameterized) semigroups $e^{-tA_{\lambda}}$ and we wish to characterize the limit:
 $\displaystyle\lim_{\lambda\to\infty}e^{-tA_{\lambda}}$.

The semigroups $e^{-tA_{\lambda}}$ also correspond to a diffusion process, with reflection at $\Gamma$, which could also get ``killed'' or ``absorbed" on
entering the region $\Omega_0$. Consequently, we expect that as $\lambda\to\infty$ this absorption occurs quicker, on average, so that (at least on a purely 
formal level) $\lambda =\infty$ corresponds to \emph{instantaneous absorption}. In other words, if we denote by $B$ the realization of the
Laplacian in $L^2(\Omega_1)$ with Neumann boundary conditions at $\Gamma$ (reflection) and Dirichlet boundary conditions at $\Gamma_0$ (instantaneous absorption),
we should then have: $\displaystyle\lim_{\lambda\to\infty}e^{-tA_{\lambda}} = e^{-tB}$.

Later, we will make these heuristics more precise. One complication is the fact that the semigroups have different domains: while $e^{-tA_{\lambda}}$ is defined on $L^2(\Omega)$, clearly $e^{-tB}$ is defined on $L^2(\Omega_1)$. We mention that the impetus to study this problem arose in the context of \emph{stochastic reaction diffusions} and comparing different mechanisms for capturing biochemical reactions. For actual applications, in addition to describing the limit, it is desirable to quantify in what manner, i.e. norm, and at what rate one gets convergence. Naturally, the choice of norm will affect the rate one obtains.
%
%
%
\begin{figure}
\begin{center}
\centering 
\scalebox{.45}{\input{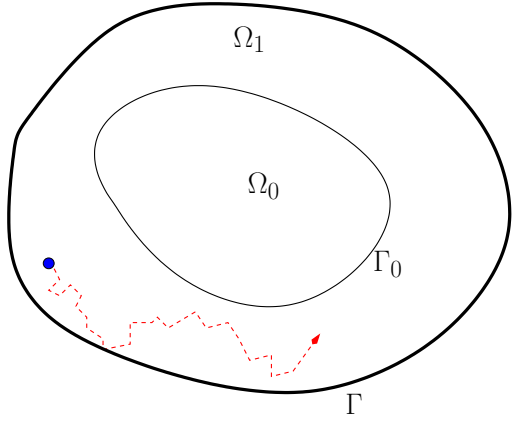}}
\end{center}
\caption{The region $\Omega$ with diffusing particle ``searching for $\Omega_0$"}
\label{fig:reaction_region}
\end{figure}

We now briefly describe our main result and methods. Our approach is to study the associated parabolic problems (see \S 2 for notation):
\begin{align} \label{eq:sobolevGenpureSmol}
 \partial_t\rho&= \kappa \lap{\rho(t,x)}, \quad (t,x)\in Q_1
\end{align}
with initial and boundary conditions
\begin{equation} \label{eq:sobolevpureSmolBC}
 \left. \begin{aligned}
 \rho(t,x) &= 0,\quad (t,x)\in\Sigma_0\\
   \nabla\rho(t,x)\cdot \vecs{\hat{n}} &= 0,\quad (t,x)\in\Sigma\\
    \at{\rho}{t=0} &= g.
  \end{aligned}
  \right\}
  \end{equation}
This problem is sometimes called the Smoluchowski problem and one usually interprets $\rho(t,x)$ as the probability density that the \emph{unreacted} diffusing particle, of diffusivity $\kappa>0$, is
located at $x\in \Omega_1$ at time $t$. Here $g(x)$ is the initial distribution of the reactant and from our previous discussion we note that
$\rho(t,x)= e^{-\kappa tB}(g(x))$.

We will also consider the following problem, which in the context of bimolecular reactions appears to be due to Doi:
 \begin{align} \label{eq:sobolevGenpureDoi}
\partial_t\pl &= \kappa \lap{\pl(t,x)} - \lambda \, \ind_{\overline{\Omega}_0}(x) \pl(t,x), \quad (t,x)\in Q
\end{align}
with initial and boundary conditions
\begin{equation}\label{eq:sobolevpureDoiBC}
   \left.\begin{aligned}
    &\at{\pl}{t=0} = \ext{g}:=
     \begin{cases}
    g(x),  &x \in \Omega_1;\\
    0, & x \in \Omega_0.
  \end{cases}\\
    & \nabla\pl(t,x)\cdot \vecs{\hat{n}} = 0,\quad\quad (t,x)\in\Sigma
 \end{aligned}
 \right\}
 \end{equation}
Let $R_if:= \at{f}{\Omega_i}$, $i=0,1$, be the restriction operators which we need in order to compare the two solutions. One of our main results is
the following:
\begin{theorem}\label{thm:main_result}
Let $0<T<\infty$ be arbitrary but fixed. Assume that all the domains have smooth boundaries and that $E_0[g] \in H^1(\Omega)$. Then there is a constant $C>0$
depending only on the domains and $\norm{g}_{H^1(\Omega_1)}$ such that
\begin{equation}\label{eq:prelim_main}
 \norm{(R_1\circ e^{-\kappa t A_{\lambda}}\circ E_0)[g]-e^{-\kappa tB}[g]}:=\norm{\at{\pl}{\Omega_1} -\rho} \leq C
\lambda^{-\frac{1}{4}},
\end{equation}
where all the norms are taken in $L^2(Q_1)$.
\end{theorem}
 This estimate turns out to be independent of the dimension $m$. Along the way, we also prove somewhat weaker results if $\Gamma_0$ is only assumed to be Lipschitz. Roughly speaking, the overall program is that we first derive estimates in the interior $Q_0$. Using ``trace theorems", we transfer these estimates to the boundary, 
$\Sigma_0$ and show that they imply convergence in the exterior region, $Q_1$.


As mentioned above, large coupling limits have already been studied. The works of {\sc demuth et al} \cite{Demuth:1980uq,Demuth:1992vu,Demuth:1993ug,M.Demuth:1995zr} contain the most complete results on the heat semigroups that we are aware of. These papers study the case $\Omega = \R^m$ which is quite different from the case considered here. Furthermore, their approach is probabilistic in nature: their main technical tools are the Feynman--Kac formula and estimates for occupation and hitting times of Brownian motion.
In \cite{M.Demuth:1995zr}, the rate obtained is, in our notation, $\mathcal{O}(\lambda^{-\frac{1}{2}+\sigma})$ for
$\norm{R_1\circ e^{- t A_{\lambda}} -e^{-t B}\circ R_1}$ in $L^2(\R^m)$. The constant $0<\sigma<\frac{1}{2}$ depends on the geometry and regularity $\Gamma_0$. If $\Gamma_0$ is smooth, they show that one can choose $\sigma$ arbitrarily close to 0 and in \cite{Demuth:1993ug} it is shown that $\sigma =1/4$ in the uniformly convex case.

A summary of this paper is as follows. In \S 2 we fix notation and collect standard facts about our main tools. We define our notion of generalized
solutions in \S 3  and derive some \emph{a priori} estimates. In \S 4, we use these estimates to show weak convergence between the two solutions. We prove Theorem~\ref{thm:main_result} in \S5, after defining and exploiting another notion of a weak solution. We conclude in \S 6 with some remarks on possible improvements of our results.

\subsection*{Acknowledgements}
This paper is a revision of part of the author's doctoral dissertation, prepared under the the direction of Professor Samuel Isaacson and presented to the Boston
University Graduate School. The author is grateful to Professor Isaacson for suggesting the problem, for his advice, encouragement as well as financial support 
while the main part of this work was completed.
The author would also like to acknowledge the comments of an anonymous reviewer which have helped improve the quality of the exposition.

\section{Preliminaries}
We pause here to fix some notation and review some facts we need later.
We define $I$ as the interval $(0,T)$ and also write 
\begin{align*}
Q:=I\times\Omega;\quad Q_0:=I\times\Omega_0; \quad Q_1:=I\times\Omega_1; \quad \Sigma_0 = I\times \Gamma_0 ;\quad \Sigma =  I \times\Gamma.
\end{align*}
$\Gamma_0$ will always be assumed to be at least a Lipschitz boundary (see below) and, for simplicity, $\Gamma$ will always be assumed smooth. We also use the slightly cumbersome notation for the ``time sliced" domains: 
$\Omega_{\{t\}}:=\{t\}\times\Omega;\,\,\Omega_{0,\{t\}}:=\{t\}\times\Omega_0;\,\, \Omega_{1,\{t\}}:=\{t\}\times\Omega_1.$
We will use the usual Sobolev spaces $H^{r}(V)$, with $V$ being any one of $\Omega_0$, $\Omega_1$ or $\Omega$, and the Sobolev spaces on cylinders 
$H^{r,s}(I\times V)=L^2(I;H^r(V))\cap H^s(I;L^2(V))$ and other Banach space valued function spaces. Facts about these spaces can be found in {\sc adams} \cite{Adams:1975fk} and {\sc lions \&  magenes} \cite{Lions:1973fk}.

We also need the ``trace theorem". We appeal to it often enough that we include a statement. Recall that a domain $V$ with boundary $\omega$ is said to be $C^{k,\gamma}$ if every point in $\omega$ has a local chart which is of H\"older class $C^{k,\gamma}$.

\begin{theorem}\label{thm:parab_trace} 
Let $k\geq1$ and let $\omega$ be a $C^{k-1,1}$ boundary and $\tfrac{1}{2}<r\leq k$, $s\geq0$, then there exists a bounded, linear, surjective ``trace map":
 \begin{align*} 
 \mathcal{T}^x_{0}: H^{r,s}(V_T) &\to H^{r-\tfrac{1}{2}, s\paren{1-\tfrac{1}{2r}}}(\omega_T)\\
 u&\mapsto\at{u}{\omega_T}
  \end{align*}
where $V_T$ is the cylinder $I\times V$ and $\omega_T = I\times \omega$ is the lateral boundary.
\end{theorem}

If the boundary is Lipschitz, {\sc costabel} \cite{Costabel:1990uq} has shown that the theorem remains true for $1/2<r<3/2$. The following variant will also be useful:
\begin{theorem}\label{thm:ModifiedTraceTheorem}
Assume $V$ is bounded and $\omega\in C^{0,1}$. Then given any $\epsilon >0$ we have, for some $C>0$ independent of $\epsilon$,
\begin{equation}\label{eq:ModTraceTheorem}
\norm{\mathcal{T}^x_0[u]}_{ L^2(\omega_T)} \leq C\paren{\epsilon^{-1}\norm{u}_{L^2(V_T)} + \epsilon\norm{Du}_{L^2(V_T)}}.
\end{equation}
\end{theorem}

Define the space $H^{1}_o(\Omega_1) := \{u\in H^{1}(\Omega_1):\at{u}{\Gamma_0} =0\}$, with the boundary value taken in the ``trace sense". Note that the ``extension by zero" operator $E_0$ is bounded from $H^{1}(\Omega_1)$ to $H^{1}(\Omega)$ exactly on the subspace $H^{1}_o(\Omega_1)$. The notation $\vecs{\hat{n}}$ is reserved for the outer normal to the boundary. We also require basic facts about interpolation of Sobolev spaces which can be found in \cite{Adams:1975fk, Lions:1973fk} (a good, concise account is also given in {\sc m{\scriptsize c}lean} \cite{Mclean:2000kx}). Finally, $C$ will denote a positive constant, usually different at different instances, which is independent of $\lambda$.


\section{Weak Solutions and \emph{A priori} Estimates}
We turn now to the definitions and estimates which will be essential in what follows. In this section, $\Gamma_0$ is Lipschitz while $\Gamma$ is smooth.
\subsection{Weak Solutions}

On a Lipschitz domain the classical Green's formulas hold (see {\sc costabel} \cite{Costabel:1990uq}). This is the basis for the following weak formulation of the initial-boundary value problem:
\begin{definition}
We say that $\rho \in H_{o}^{1,0}(Q_1)$ is a generalized solution of \eqref{eq:sobolevGenpureSmol} and \eqref{eq:sobolevpureSmolBC} if
 for all $\psi\in H^{1,1}(Q_1)\equiv H^{1}(Q_1)$ with $\at{\psi}{\Sigma} =0$ and  $\at{\psi}{\Omega_{1,\{T\}}} =0$.
 \begin{equation*}
 \int_{Q_1}{\paren{\kappa\nabla\rho\cdot\nabla \psi -\rho\partial_t{\psi}}} -\int_{\Omega_{1,\{0\}}}{g\psi} =0;
\end{equation*}
\end{definition}


Using the Lion's Projection Lemma (for a statement see \cite[Lemma 2.1]{Costabel:1990uq}), one can show the existence and uniqueness of the weak solution 
$\rho$ as defined above. The argument is very similar to that in \cite[Lemma 2.3]{Costabel:1990uq} and can be found in {\sc agbanusi} \cite{Agbanusi:2016ibp}.
\begin{definition}\label{defn:peturb_weak_soln}
We say that  $\pl \in H^{1,0}(Q)$ is a generalized solution of \eqref{eq:sobolevGenpureDoi} and \eqref{eq:sobolevpureDoiBC} if
 for all $\phi\in H^{1}(Q)$ with  $\at{\phi}{\Omega_{\{T\}}} =0$.
\begin{equation*}
 \int_{Q}{\paren{\kappa\nabla\pl\cdot\nabla \phi-\pl\partial_t{\phi}}} -\int_{\Omega_{\{0\}}}{\ext{g}\phi} +\lambda\int_{Q_0}\pl\phi =0.
  \end{equation*}

\end{definition}

A crucial fact is that we can also think of the Doi problem as a transmission problem as a simple integration by parts argument, which we omit, shows. (See  {\sc girsanov} 
\cite{Girsanov:1960dc}, {\sc ladyzhenskaya et al} \cite{Ladyzhenskaya:1966kx} and {\sc Olenik} \cite{Olenik:1959dc}):
\begin{lemma}
 Let $\pl^{+} =\at{\pl}{\Omega_0}$ and $\pl ^{-}=\at{\pl}{\Omega_1}$. Then the solution to the Doi problem can be obtained by finding 
$\pl^{+} \in H^{1,0}(Q_0)$ and $\pl^{-} \in H^{1,0}(Q_1)$ such that
\begin{equation}\label{eq:PieceGenpureDoi}
\left. \begin{aligned} 
\partial_t\pl &= \kappa\lap{\pl^{+}(t,x)} - \lambda \, \pl^{+}(t,x), \quad (t,x)\in Q_0,\\
\partial_t\pl &=  \kappa \lap{\pl^{-}(t,x)}, \quad (t,x) \in Q_1, 
\end{aligned}
\right\}
\end{equation}
with the transmission conditions,
\begin{equation}\label{eq:coup_cond}
\left. \begin{aligned}
\pl^{+}(t,x) &=\pl^{-}(t,x),\quad (t,x)\in \Sigma_0,\\
\nabla\pl^{+}(t,x)\cdot\vecs{\hat{n}}&= \nabla\pl^{-}(t,x)\cdot\vecs{\hat{n}}, \quad (t,x)\in \Sigma_0,
\end{aligned}
\right \}
\end{equation}
and the external boundary condition $\nabla\pl^{-}(t,x)\cdot\vecs{\hat{n}}=0$, for $x\in\Sigma$ as well as the initial conditions $\pl^{+}(0,x) =0$
and $\pl^{-}(0,x) =g(x)$.
\end{lemma}
The papers \cite{Girsanov:1960dc,Ladyzhenskaya:1966kx,Olenik:1959dc} also contain existence/uniqueness results for fixed $\lambda$.

\subsection{Energy Estimates}
This section is devoted to some simple integral estimates which are at the heart of all the results in this note. The first is
\begin{lemma}[Uniform $L^2$ Bounds]
\label{lem:UniformBounds}
Let $\pl(t,x)$ satisfy \eqref{eq:sobolevGenpureDoi}. Given initial condition $\ext{g}(x)$ in $L^2(\Omega)$, then $\pl$ and $\displaystyle \nabla{\pl}(t,x)$ are uniformly bounded 
in $L^{\infty}(I;L^2(\Omega))$ and  $L^2(Q)$ respectively.
In particular, $\norm{\pl}_{L^2(Q_0)} \to 0$ as $\lambda \to \infty$.
\end{lemma}

\begin{proof}
Multiply \eqref{eq:sobolevGenpureDoi} by $\pl(t,x)$ and integrate over $\Omega$ to obtain
\begin{equation*}
\frac{1}{2}\DO{t}\int_{\Omega}{\pl^{2}~dx} = \kappa\int_{\Omega}{\pl\Delta\pl~dx}-\lambda\int_{\Omega_0}{\pl^{2}~dx}.
\end{equation*}
Integrating by parts on the second term, applying boundary conditions on $\partial\Omega$ followed by an integration from 0 to $T$ in $t$ gives
\begin{equation}\label{eq:MainUniformBound}
\frac{1}{2}\norm{\pl(t,\cdot)}_{L^2(\Omega)}^{2} + \kappa\norm{\nabla\pl}^{2}_{L^2(Q)} +\lambda\norm{\pl}^{2}_{L^2(Q_0)} = \frac{1}{2}\norm{\pl(0,\cdot)}^{2}_{L^2(\Omega)}.
\end{equation}
Since $\pl(0,x):= \ext{g}(x)$, it follows that
\begin{align*}
\sup_{t\in I}{\norm{\pl(t)}^{2}_{L^2(\Omega)}}\leq K_1,\quad
\norm{\nabla\pl}^{2}_{L^2(Q)}\leq K_2,\quad
\norm{\pl}^{2}_{L^2(Q_0))}\leq\frac{K_1}{2\lambda}.
\end{align*}
Here $\displaystyle K_1 = \norm{g(x)}^{2}_{L^2(\Omega_1)}$ and $K_2 =K_{1}/2 \kappa$. Thus $\pl(t,x)$ and $\displaystyle \nabla{\pl}(t,x)$ are uniformly bounded while  $\norm{\pl}_{L^2(Q_0)} = \mathcal{O}(\lambda^{-\frac{1}{2}})$, as claimed.
\end{proof}

If we impose additional regularity in the initial condition $g$ we can prove:

\begin{lemma}
\label{lem:MoreUniformBounds}
Suppose now that the initial condition $\ext{g}(x)$ is in $H^1(\Omega)$. 
Then for a.e. $t>0$, $\partial\pl/\partial t$ and $ \nabla\pl(t,\cdot)$ are uniformly bounded in $L^2(Q)$ and $L^2(\Omega)$ 
respectively. Moreover, $\displaystyle \norm{\pl(t,\cdot)}_{L^2(\Omega_0)}\to 0$ as $\lambda \to \infty$, for a.e. $t\in I$.
\end{lemma}

\begin{proof}
The proof is similar to that above except for a mild technicality. We multiply equation ~\eqref{eq:sobolevGenpureDoi} by $\partial\pl/\partial t$ and integrate by parts
over $\Omega$ to obtain (henceforth, we drop the integration measures)
\begin{equation}\label{eq:step}
\int_{\Omega}{|\partial_t\pl|^2} =  -\kappa\sum_{i}\int_{\Omega}{\PD{\pl}{x_i}\PDDM{\pl}{x_i}{t}}-\lambda\int_{\Omega_0}{\pl\partial_t\pl}.
\end{equation}
A theorem of {\sc olenik} \cite{Girsanov:1960dc, Olenik:1959dc} allows us to switch the order of differentiation so  
\begin{equation*}
\sum_{i}\int_{\Omega}{\PD{\pl}{x_i}\PDDM{\pl}{x_i}{t}}= \sum_{i}\int_{\Omega}{\PD{\pl}{x_i}\PDDM{\pl}{t}{x_i}} =\frac{1}{2}\DO{t}\int_{\Omega}{\abs{\nabla\pl}^2}.
\end{equation*}
Hence it follows that \eqref{eq:step} becomes
\begin{equation*}
2\norm{\partial_t\pl(t,\cdot)}^2_{L^2(\Omega)}=-\DO{t}\paren{\kappa\norm{\nabla\pl(t,\cdot)}^{2}_{L^2(\Omega)}+\lambda\norm{\pl(t,\cdot)}^2_{L^2(\Omega_0)}}.
\end{equation*}
An integration in $t$ shows the following estimates which prove the result:
\begin{align*}
\sup_{I}\norm{\nabla\pl(t,\cdot)}^{2}_{L^2(\Omega)}\leq C_1,\quad
\norm{\partial_t\pl}^2_{L^2(Q)}\leq C_2,\quad
\sup_{I}\norm{\pl(t,\cdot)}^2_{L^2(\Omega_0)}\leq\frac{C_3}{\lambda}.
\end{align*}
\end{proof}

\section{Weak Convergence}
The preceding estimates---combined with interpolation, duality and the trace theorem---have some simple consequences  which we explore in this section. The assumptions of \S3 still hold. In particular, $\Gamma_0$ is Lipschitz and $\Gamma$ is smooth.
We begin with:
\begin{lemma}\label{lem:weak_convergence}
There exist $p^*$ and a subsequence of $\{\pl\}$ such that $\pl\rightharpoonup p^*$ weakly in $H^{1,1}(Q)$ and $\pl \rightharpoonup p^*$ 
weak--$\ast$ in $L^{\infty}(I;H^1(\Omega))$.
\end{lemma}
 
 \begin{proof}
Lemmas \ref{lem:UniformBounds} and \ref{lem:MoreUniformBounds} imply that $\pl$ is uniformly bounded in $L^{\infty}(I;H^1(\Omega))$. 
Thus weak--$\ast$ convergence follows by duality. Now as $T$ is finite, we have the elementary embedding $L^{\infty}(I;L^2(\Omega))\hookrightarrow L^{2}(Q)$.
This implies that $\pl$ is also a uniformly bounded sequence in $H^1(I;L^2(\Omega))\, \cap\, L^{2}(I;H^1(\Omega))$ as well.  The existence of a weak limit follows since bounded sequences in a Hilbert space have a weakly convergent subsequence.
 \end{proof}

\begin{remark} 
By the Sobolev embedding theorem, $\pl$ is also uniformly bounded in $C^{0}([0,T];L^2(\Omega))$.
\end{remark} 

Next, we show that in $Q_0$ one can get convergence in more regular Sobolev spaces. Recall that $\pl^{+}=\at{\pl}{Q_0}$.

\begin{lemma}\label{lem:interior_higher_conv}
Let $0\leq \epsilon_1, \epsilon_2\leq1$ be fixed and let $2\epsilon_0:= \min \paren{1-\epsilon_1,1-\epsilon_2}$. Then $\norm{\pl^{+}}_{H^{\epsilon_1,\epsilon_2}(Q_0)} =\mathcal{O}(\lambda^{-\epsilon_0})$.
 If in addition $1/2<\epsilon_1\leq1$, then $\norm{\mathcal{T}_0^x{\pl^{+}}}_{H^{\sigma_1,\sigma_{2}}(\Sigma_0)} =\mathcal{O}(\lambda^{-\epsilon_0})$ with $\sigma_1 :=\epsilon_1-1/2$ and $\sigma_2 :=\epsilon_{2}(1-1/2\epsilon_1)$.

\end{lemma}

\begin{proof}
Since $\pl^{+}$ is uniformly bounded in $H^{1,1}(Q_0)$ we have by interpolation that
\begin{align*}
\norm{\pl^{+}}_{H^{\epsilon_{2}}(I;L^2(\Omega_0))} &\leq C \norm{\pl^{+}}^{1-\epsilon_2}_{L^2(Q_0)}\,\,\,\norm{\pl^{+}}^{\epsilon_2}_{H^{1}(I;L^2(\Omega_0))}
\leq C\lambda^{-(1-\epsilon_2)/2}.
\end{align*}
Similarly we have
\begin{align*}
\norm{\pl^{+}}_{ L^2(I;H^{\epsilon_1}(\Omega_0))} &\leq C \norm{\pl^{+}}^{1-\epsilon_1}_{L^2(Q_0)}\,\,\,\norm{\pl^{+}}^{\epsilon_1}_{L^{2}(I;H^1(\Omega_0))}
\leq C\lambda^{-(1-\epsilon_1)/2}.
\end{align*}
Hence
\begin{align*}
\norm{\pl^{+}}_{H^{\epsilon_1,\epsilon_{2}}(Q_0)}^2 &= \paren{\norm{\pl^{+}}_{L^2(I;H^{\epsilon_1}(\Omega_0))}^2 + \norm{\pl^{+}}_{H^{\epsilon_{2}}(I;L^2(\Omega_0))}^2}
 = \mathcal{O}(\lambda^{-2\epsilon_0}),
\end{align*}
 which proves the first part. A direct application of the trace theorem then completes the proof.
 
\end{proof}
Having established a weak limit, $p^{*}$, we show that it coincides with $\rho$. 
\begin{proposition}
 We have that $\pl\rightharpoonup\ext{\rho}$ weakly.
\end{proposition}
\begin{proof}
Fix any $0\leq\sigma_1,\sigma_2<\frac{1}{2}$. Lemma~\ref{lem:interior_higher_conv} implies that $\norm{\mathcal{T}^x_0p^{*}}_{H^{\sigma_1,\sigma_2}(\Sigma_0)} = 0$. As $\pl^{+}$ is uniformly bounded in $H^{1,1}(Q_0)$ it follows, by the trace theorem, that $\at{\pl}{\Sigma_0}$
 is uniformly bounded in $H^{\frac{1}{2},\frac{1}{2}}(\Sigma_0)$. However, $H^{\frac{1}{2},\frac{1}{2}}(\Sigma_0)\hookrightarrow H^{\sigma_1,\sigma_2}(\Sigma_0)$ embeds
 compactly and as such $\at{p^{*}}{\Sigma_0} =0$ in $H^{\frac{1}{2},\frac{1}{2}}(\Sigma_0)$ or, equivalently, 
$p^{*} \in H_{o}^{1,0}(Q_1)$.
 
Now for any  $\psi \in H^1(Q_1)$ with $\at{\psi}{\Sigma_0} =\at{\psi}{\Omega_{1,\{T\}}} =0$, let $\phi =\ext{\psi}$. Using the existence of a weak limit and taking limits in 
Definition \ref{defn:peturb_weak_soln}, we get
\begin{align*}
0&= \int\limits_{Q_1}{\paren{\kappa\nabla p^{*}\cdot\nabla \psi-p^{*}\pd{\psi}{t}}} -\int\limits_{\Omega_{1,\{0\}}}{g\psi}.
  \end{align*}
  Hence $p^{*}$ is a weak solution to \eqref{eq:sobolevGenpureSmol} and \eqref{eq:sobolevpureSmolBC} and by uniqueness $p^{*} =\rho$.
\end{proof}

Observe that this weak convergence, in view of the Rellich's compactness theorem, implies strong convergence in $L^2(Q_1)$ although with no rate. We end this section with a slight sharpening of Lemma~\ref{lem:interior_higher_conv} which we shall use in the next section to derive a rate for the  $L^2(Q_1)$ convergence.

\begin{lemma}[$L^2$ boundary convergence]\label{lem:optimal_boundary}
We have $\norm{\pl}_{L^2(\Sigma_0)} = \mathcal{O}(\lambda^{-\frac{1}{4}})$.
\end{lemma}

\begin{proof}
By the trace theorem, \[\norm{\pl}_{L^2(\Sigma_0)} \leq C( \epsilon^{-1}\norm{\pl}_{L^2(Q_0)} + \epsilon\norm{\nabla\pl}_{L^2(Q_0)}).\]
Recalling that Lemmas \ref{lem:UniformBounds} and \ref{lem:MoreUniformBounds} imply $\norm{\nabla\pl}_{L^2(Q_0)}\leq C$ and $\norm{\pl}_{L^2(Q_0)}\leq C\lambda^{-1/2}$, picking $\epsilon=\lambda^{-r}$ for some $r>0$, this trace estimate becomes
\[
\norm{\pl}_{L^2(\Sigma_0)}\leq C(\lambda^{r-\frac{1}{2}}+\lambda^{-r}).
\]
Finally, the choice $r=1/4$ balances the two quantities in brackets, completing the proof.
\end{proof}
Formally, we could take $\epsilon_1 = 1/2$ in Lemma \ref{lem:interior_higher_conv} to get 
precisely the same estimate in Lemma \ref{lem:optimal_boundary}. However, the trace map ceases to be a bounded surjection when $\epsilon_1 =1/2$, hence the roundabout
argument.

\section{Strong Convergence}
So far, we have shown that $\pl\rightharpoonup E_0[\rho]$ weakly and $\at{\pl}{Q_0} \to 0$ strongly. We would like to ``transfer" this strong convergence on the interior $Q_0$
to the exterior domain, $Q_1$. 
We start by defining the error, $e$ between the two solutions in $Q_1$: $e(t,x):= \pl^{-}(t,x) -\rho(t,x)$. It follows that $e$ satisfies:
\begin{align}
\partial_t{e} &= \kappa\lap{e}, \quad (t,x) \in Q_1,
\end{align}
\begin{equation}
\left.\begin{aligned}
&e(t,x) =\pl^{+}(t,x), &&\quad (t,x) \in \Sigma_0,\\
&\nabla e(t,x)\cdot\vecs{\hat{n}} =0, &&\quad (t,x) \in \Sigma,
\end{aligned}
\right\}
\end{equation}
with the initial condition $e(0,x)=0$. Thus $e$ satisfies a homogenous heat equation in $Q_1$ and is  basically controlled by the behavior of $\pl^{+}$
on $\Sigma_0$. Since $\at{\pl^{+}}{\Sigma_0} \to 0$ we should get that $e\to0$.

The classical theory for such equations requires the Dirichlet data to be
in $H^{\frac{1}{2},\frac{1}{4}}(\Sigma_0)$. Unfortunately, Lemma~\ref{lem:interior_higher_conv} only 
shows that $\norm{\pl^{+}}_{H^{\frac{1}{2},\frac{1}{4}}(\Sigma_0)} =\mathcal{O}(1)$, while Lemma~\ref{lem:optimal_boundary} gives $\norm{\pl^{+}}_{L^2(\Sigma_0)}=\mathcal{O}(\lambda^{-1/4})$.
Our task is to extend the solvability theory for the above equations in order to include weaker boundary spaces where we happen to have stronger convergence rates. For technical reasons, in this section we assume both $\Gamma_0$ and $\Gamma$ are $C^{\infty}$.

Consider the following abstract parabolic Cauchy problem:
\begin{equation}\label{eq:abstract_smol_problem}
\left.\begin{aligned}
(\partial_t-\kappa\lap) u&=f && (t,x) \in Q_1\\
u &= h && (t,x)\in \Sigma_0 \\
\nabla u \cdot \vecs{\hat{n}}&=g &&(t,x)\in \Sigma\\
u(0,x)&= u_0(x)
\end{aligned}
\right\}
\end{equation}

 We will now give a way to define solutions with rather ``rough'' inhomogenous data $f(t,x)$, $h(t,x)$, $g(t,x)$ and $u_0(x)$. The technique is essentially a time reversal and
 transposition argument and appears to be due, in this form, to  {\sc lions \& magenes} \cite{Lions:1973fk} although the idea can be traced back to Holmgren. 
 
 First we introduce some useful Sobolev spaces. Let $\{\phi_k(x),\alpha_k\}$ be eigenpairs---normalized eigenfunctions and eigenvalues---in $L^2(\Omega_1)$
  satisfying
 \begin{align*}
-\lap\phi_k =\alpha_k\phi_k;\quad
\at{\phi_k}{\Gamma_0}=0;\quad
\at{\nabla \phi_k\cdot \vecs{\hat{n}}}{\Gamma} =0
\end{align*}
Define $\dot{H}^{r}(\Omega_1):=\{u\in L^2(\Omega_1)\,:\, \sum_k\alpha_k^{r}\abs{(u,\phi_k)}^2 < \infty\}$
and put $\dot{H}^{r,s}(Q_1) :=L^2(I;\dot{H}^{r}(\Omega_1))\cap {H}^{s}(I;L^2(\Omega_1))$. A characterization of these spaces is given in {\sc agbanusi} \cite{Agbanusi:2016ibp}.
In particular, it is shown that $\dot{H}^{1}(\Omega_1)={H}^{1}_{o}(\Omega_1)$.

Let us define the operator $\Lambda:= \partial_t-\kappa\lap$ whose \emph{formal} transpose is $\Lambda^{t}:= -\partial_t-\kappa\lap$. The following result is proved
in \cite{Agbanusi:2016ibp}:

\begin{proposition}\label{prop:boundedness_of_transpose}
For $v\in L^2(Q_1)$, there exists a unique $w:=\mathcal{S}(v)\in \dot{H}^{2,1}(Q_1)$ satisfying 
\[\Lambda^{t}w =v;\quad \at{w}{\Sigma_0}=0;\quad \at{\nabla w\cdot \vecs{\hat{n}}}{\Sigma}=0;\quad w(T,x)=0,\]
and such that the following estimates holds for some $C>0$:
\[
\norm{w}_{\dot{H}^{2,1}(Q_1)}\leq C\norm{v}_{L^2(Q_1)}.
\]
\end{proposition}

This leads us to the following definition
\begin{definition}
Suppose that the data $f(t,x)$, $h(t,x)$, $g(t,x)$ and $u_0(x)$ of \eqref{eq:abstract_smol_problem} are in $(\dot{H}^{2,1}(Q_1))^{*}$, $H^{-\frac{1}{2},-\frac{1}{4}}(\Sigma_0)$, $H^{-\frac{3}{2},-\frac{3}{4}}(\Sigma)$
and $\dot{H}^{-1}(\Omega_1)$ the dual space of $\dot{H}^{2,1}(Q_1)$, $H^{\frac{1}{2},\frac{1}{4}}(\Sigma_0)$, $H^{\frac{3}{2},\frac{3}{4}}(\Sigma)$ and $\dot{H}^{1}(\Omega_1)$ respectively.  
Let the linear functional:
\begin{equation*}
\mathcal{B}(v):= \int\limits_{Q_1}f(t,x)w(t,x) +\int\limits_{\Sigma_0}h(t,x)\partial_{\vecs{\hat{n}}}w(t,x) + \int\limits_{\Sigma}g(t,x){w}(t,x)+\int\limits_{\Omega_1}u_0(x)w(0,x), 
\end{equation*}
where $w = \mathcal{S}(v)$ is the unique solution guaranteed by Proposition~\ref{prop:boundedness_of_transpose},
we say that $u\in L^2(Q_1)$ is a \textit{very weak} solution of \eqref{eq:abstract_smol_problem} if for all $v \in L^2(Q_1)$
\begin{align}
\mathcal{B}(v): = \int_{Q_1}u(t,x)v(t,x).
\end{align}
\end{definition}
The next result is a direct consequence of the preceding definition. It is established by showing that $\mathcal{B}(v)$ is a bounded linear functional on 
$L^2(Q_1)$ and applying the Riesz representation theorem.

\begin{proposition}\label{prop:existence_very_weak}
The equations  \eqref{eq:abstract_smol_problem} admit a unique \emph{very weak} solution $u \in L^2(Q_1)$ satisfying
\begin{equation}\label{eq:very_weak_est}
 \norm{u}_{ L^2(Q_1)}  \leq C\paren{\norm{u_0}_{\dot{H}^{-1}(\Omega_1)}+ \norm{f}_{\dot{H}^{-2,-1}(Q_1)}  + \norm{h}_{H^{-\frac{1}{2},-\frac{1}{4}}(\Sigma_0)}+\norm{g}_{H^{-\frac{3}{2},-\frac{3}{4}}(\Sigma)}}.
\end{equation}
\end{proposition}
A complete proof can be found in \cite{Agbanusi:2016ibp} and we stress that this proof relies on the smoothness of $\Gamma_0$. Armed with this, we turn to the
\begin{proof}[{\bf Proof of Theorem~\ref{thm:main_result}}.]
Simply apply  \eqref{eq:very_weak_est} of Proposition~\ref{prop:existence_very_weak} to $e:= \pl^{-}(t,x) -\rho(t,x)$ with $f=0$, $u_0 =0$, $g=0$ and $h:=\at{\pl^+}{\Sigma_0}\in L^2(\Sigma_0)\subset H^{-\frac{1}{2},-\frac{1}{4}}(\Sigma_0)$.
\end{proof}
Interpolating between the $L^2(Q_1)$ of  bounds of Theorem~\ref{thm:main_result} and the uniform $H^{1,1}(Q_1)$ bounds of \S3 gives
\begin{corollary}
Fix $\delta_1, \delta_2 \in (0,1)$ and define $4\delta_0:= \min \paren{1-\delta_1,1-\delta_2}$. Then we have $\norm{\pl^{-}-\rho}_{H^{\delta_1,\delta_2}(Q_1)} =\mathcal{O}(\lambda^{-\delta_0}).$
\end{corollary}


\section{Final Remarks, Open Problems}


The rate we obtained is not the one conjectured to be sharp in the case where the inclusion has a smooth boundary. However, these conjectures appear to be based on the case where $\Omega =\R^m$, i.e., when there is no external boundary. We have not successfully worked out the \emph{precise} effect of the external boundary. That there could be extra difficulty is suggested by examining the corresponding large coupling problem for the \emph{acoustic wave equation}. If there are no external boundaries and the obstacle is strictly convex, then the rays of geometric optics interact at most once with $\Gamma_0$ if they are not energetic enough compared to $\lambda$ (or if they just graze $\Gamma_0$). This is no longer true if there is an external boundary due to multiple reflections and trapping. Things get even more complicated if one considers high frequency rays\ldots

In any case, here are some possible paths  towards improving the results of this paper.
Our proof rests on the energy estimates, the trace theorems and Proposition \ref{prop:existence_very_weak}. The energy estimates derived here only assumed that the boundary was Lipschitz. In general, energy estimates rely on choosing ``multipliers" and it is possible other choices of multipliers, perhaps better reflecting the geometry/regularity of the boundary, might give stronger estimates. 

For arbitrary $u\in H^1$ the trace theorem says that the restriction to the boundary is in $H^{1/2}$. However, this can be improved if, in addition, $u$ satisfies
$P(x,D)u \in L^2$ for certain second order operators $P$. For the improvement when $P$ is hyperbolic, see {\sc tataru} \cite{Tataru:1998fk}. Thus, there remains the possibility of obtaining sharper trace theorems better suited to the problem at hand.

We do not know the analogue---if there is one---of Proposition \ref{prop:existence_very_weak} or \eqref{eq:very_weak_est} when $\Gamma_0$ is Lipschitz and the Dirichlet data is assumed to be in $L^2(\Sigma_0)$. The extension of the proposition to this case would imply that our main theorem, Theorem \ref{thm:main_result}, holds for Lipschitz domains.

Finally, let us mention the possibility of constructing good approximate solutions i.e. parametrices near the interface $\Gamma_0$ which may yet yield better estimates but require much more sophisticated tools than those employed here.

\thispagestyle{empty}
\bibliography{CPDE Revision 1.bbl} 
\bibliographystyle{amsplain}

\end{document}